\documentclass[reqno]{amsart}
\usepackage[utf8]{inputenc}
\usepackage[T1]{fontenc}
\usepackage{pgf,pgfarrows,pgfnodes,pgfautomata,pgfheaps,pgfshade,hyperref, amssymb}
\usepackage{amssymb}
\usepackage{enumitem}
\usepackage[english]{babel}
\usepackage[capitalize]{cleveref}
\usepackage{mathtools,tikz}
\usepackage[colorinlistoftodos]{todonotes}
\usepackage{soul}

\usepackage{tikz}

\usepackage{xcolor}

\hypersetup{
    colorlinks,
    linkcolor={blue!30!black},
    citecolor={green!50!black},
    urlcolor={blue!80!black}
} 

\usepackage{mathrsfs} 
\usepackage{dsfont}

\newtheorem{theorem}{Theorem}[section]
\newtheorem{proposition}[theorem]{Proposition}

\newtheorem{lemma}[theorem]{Lemma}

\newcounter{thmcounter}

\newcounter{introthmcounter}

\newtheorem{corollary}[theorem]{Corollary}
\theoremstyle{definition}

\newtheorem*{definition*}{Definition}

\newtheorem*{question*}{Question}
\newcounter{proofcount}
\AtBeginEnvironment{proof}{\stepcounter{proofcount}}

\makeatletter                  
\@addtoreset{claim}{proofcount}
\makeatother

\theoremstyle{remark}

\newtheorem{remark}[theorem]{Remark}

\newtheorem*{remark*}{Remark}
\newtheorem*{example*}{Example}

\def\Z{{\mathbb Z}}

\def\N{{\mathbb N}}

\def\cE{{\mathcal E}}
\def\cB{{\mathcal B}}
\def\cC{{\mathcal C}}
\def\cA{{\mathcal A}}
\def\cL{{\mathcal L}}

\def\cM{{\mathcal M}}

\def\cS{{\mathcal S}}
\def\cX{{\mathcal X}}

\def\freq{{\rm freq}}

\newcommand{\1}{\ensuremath{\mathds{1}}}

\title[Multiple partial rigidity rates in low complexity subshifts]{Multiple partial rigidity rates in low complexity subshifts}

\author{Trist\'an Radi\'c}
\address{Department of mathematics, Northwestern University, 2033 Sheridan Rd, Evanston, IL, United States of America}
\email{tristan.radic@u.northwestern.edu}

\thanks{Northwestern University}



\begin{document}
\date{\today}
\maketitle


\begin{abstract}
Partial rigidity is a quantitative notion of recurrence and provides a global obstruction which prevents the system from being strongly mixing. A dynamical system $(X, \cX, \mu, T)$ is partially rigid if there is a constant $\delta >0$ and sequence $(n_k)_{k \in \N}$ such that $\displaystyle \liminf_{k \to \infty } \mu(A \cap T^{n_k}A) \geq \delta \mu(A)$  for every $A \in \cX$, and the partial rigidity rate is the largest $\delta$ achieved over all sequences. 
For every integer $d \geq 1$, via an explicit construction, we prove the existence of a minimal subshift $(X,S)$ with $d$ ergodic measures having distinct partial rigidity rates. The systems built are $\cS$-adic subshifts of finite alphabetic rank that have non-superlinear word complexity and, in particular, have zero entropy. 
\end{abstract}

\section{Introduction}

For measure preserving systems, partial rigidity quantitatively captures recurrence along a particular trajectory. Roughly speaking, this measurement ensures that at least a proportion $\delta \in (0,1]$ of any measurable set $A$ returns to $A$ along some sequence of iterates. The notion was introduced by Friedman \cite{Friedman_partial_mixing_rigidity_factors:1989} and defined formally by King \cite{King_joining-rank_finite_mixing:1988}. An important property of partially rigid systems is that, besides the trivial system, they are not strongly mixing. Although the converse does not hold, many common examples of non-mixing systems are partially rigid, see for example \cite{Dekking_Keane_mixing_substitutions:1978,Katok_interval_exchange_not_mixing:1980,Cortez_Durand_Host_Maass_continuous_measurable_eigen_LR:2003,Bezuglyi_Kwiatkowski_Medynets_Solomyak_Finite_rank_Bratteli:2013,Danilenko_finite_rank_rationalerg_partial_rigidity:2016,Creutz_mixing_minimal_comp:2023, Goodson_Ryzhikov_conj_joinings_producs_rank1:1997}. 

To be more precise, a measure-preserving systems $(X, \cX, \mu, T)$ is \emph{partially rigid} if there exists $\delta > 0$ and an increasing sequence $(n_k)_{k \in \N}$ of integers such that 
\begin{equation}  \label{eq p rigid}
\liminf_{k \to \infty} \mu (A \cap T^{-n_k}A) \geq \delta \mu(A)
\end{equation}
for every measurable set $A$. A constant $\delta>0$ and a sequence $(n_k)_{k \in \N}$ satisfying \eqref{eq p rigid} are respectively called a \emph{constant of partial rigidity} and a \emph{partial rigidity sequence}. 

Once we know that a system is partially rigid, computing the largest value of $\delta$ provides valuable information on how strongly the system exhibits recurrent behavior. In particular, as was remarked by King in 1988 \cite[Proposition 1.13]{King_joining-rank_finite_mixing:1988}, this constant is invariant under measurable isomorphisms and increases under factor maps. We call this constant the \emph{partial rigidity rate}, we denote it $\delta_{\mu}$ and it is given by 
\begin{equation*}
    \delta_{\mu} = \sup \{ \delta >0 \mid \delta \text{ is a partial rigidity constant for some sequence } (n_k)_{k \in \N} \},
\end{equation*}
with the convention that $\delta_{\mu} = 0$ whenever the system is not partially rigid. 
There are only limited partially rigid systems for which that constant is known. One major case is \emph{rigid systems}, that is when $\delta_{\mu}=1$. Such systems have been well studied after Furstenberg and Weiss introduced them in \cite{Furstenberg_Weiss77}, see for instance \cite{Bergelson_delJunco_Lemanczyk_Rosenblatt_rigidity_nonrecurrence:2014,Coronel_Maass_Shao_seq_entropy_rigid:2009,Donoso_Shao_uniform_rigid_models:2017,Fayad_Kanigowski_rigidity_wm_rotation:2015,Glasner_Maon_rigidity_topological:1989}. 
The only non-rigid examples for which the partial rigidity rates are calculated are some specific substitution subshifts studied in \cite[Section 7]{donoso_maass_radic2023partial}.

Since minimal substitution subshifts are uniquely ergodic, it is natural to ask whether it is possible to construct a minimal, low-complexity system with more than one ergodic measure and distinct partial rigidity rates. Via an explicit construction, we fully resolve this question. More precisely, we show
\begin{theorem} \label{main thrm}
    For any natural number $d\geq 2$, there exists a minimal subshift with non-superlinear complexity that has $d$ distinct ergodic measures $\mu_0, \ldots, \mu_{d-1}$ for which the partial rigidity rates $0< \delta_{\mu_0} < \ldots < \delta_{\mu_{d-1}} < 1$ are also distinct. 
    
    Moreover, the partial rigidity sequence $(n_k)_{k \in \N}$ associated to each $\delta_{\mu_i}$ is the same for all $i \in \{0,\ldots, d-1\}$. 
\end{theorem}
Constructing measures all of which share the same partial rigidity sequence is a key aspect because, in general, an invariant measure can be partially rigid for two different sequences $(n_k)_{k \in \N}$ and $(n'_k)_{k \in \N}$ and have different partial rigidity constants $\delta$ and $\delta'$ for each sequence. For instance, in \cite[Theorem 7.1]{donoso_maass_radic2023partial} it is proven that for the Thue-Morse substitution subshift equipped with its unique invariant measure $\nu$, $\delta_{\nu} = 2/3$ and its associated partial rigidity sequence is $(3 \cdot 2^n)_{n \in \N}$. Using a similar proof, the largest constant of partial rigidity for the sequence $(2^n)_{n \in \N}$ is $1/3$. In contrast, the discrepancy between the values in \cref{main thrm} is not due to quantifying along a different trajectory, but rather that for each measure the returning mass takes on a different value.

The system constructed to prove \cref{main thrm} is an $\cS$-adic subshift, that is a symbolic system formed as a limit of morphisms $\boldsymbol \sigma = (\sigma_n \colon A_{n+1}^* \to A_n^*)_{n \in \N}$ (see \cref{section prelimanries} for the precise definitions). We introduce a novel technique that allows us to build minimal $\cS$-adic subshift with $d$ ergodic measures, where each ergodic measure ``behaves like'' a substitution subshift for which we already know its partial rigidity rate. The idea is that the measures of the cylinder sets ``closely approximate'' the values assigned by the unique invariant measure of the substitution subshift that is ``imitating''. For the precise statement, see \cref{thrm gluing technique}. This gluing technique is of interest on its own, as it gives a general way for controlling distinct ergodic measures in some specific $\cS$-adic subshift.

For each ergodic measure $\mu_i$, with $i \in \{0,\ldots,d-1\}$, the gluing technique gives us a lower bound for the partial rigidity rate (see \cref{cor delta smaler}). The lower bound corresponds to the partial rigidity rate associated to the uniquely ergodic system that the measure $\mu_i$ is ``imitating''. In \cref{section computation partial rigidity}, we restrict to a specific example in which that lower bound is achieved. In that section, we prove that the number of morphisms needed for building the $\cS$-adic subshift can be reduced to three. 
Combining results from Sections \ref{section gluing technique} and \ref{section computation partial rigidity}, we complete the proof of \cref{main thrm}. An extended version of the theorem that includes the values of $\delta_{\mu_i}$ for $i \in \{0, \ldots,d-1\}$ and the partial rigidity sequence is stated in \cref{thrm final result}.

\textbf{Acknowledgments.}
The author thanks B. Kra for her careful reading and helpful suggestions on the earlier versions of this paper. He is also grateful to A. Maass and S. Donoso for their insights in the early stages of this project, and extends his thanks to F. Arbulu for providing valuable references. Special thanks to S. Petite, who, during the author's first visit to the UPJV in Amiens, asked whether an example with multiple partial rigidity rates, such as the one described in this paper, could be constructed.

\section{Preliminaries and notation} \label{section prelimanries}


\subsection{Topological and symbolic dynamical systems}

In this paper, a {\em topological dynamical system} is a pair $(X,T)$, where $X$ is a compact metric space and $T \colon X \to X$ is a homeomorphism. We say that $(X,T)$ is {\em minimal} if for every $x \in X$ the orbit $\{T^n x: n\in \Z\}$ is dense in $X$. A continuous and onto map $\pi \colon X_1 \to X_2$  between two topological dynamical systems $(X_1, T_1)$ and $(X_2,T_2)$ is a \emph{factor map} if for every $x \in X_1$, $T_2 \circ \pi (x) = \pi \circ T_1 (x) $. 

We focus on a special family of topological dynamical system, symbolic systems. To define them, let $A$ be a finite set that we call {\em alphabet}. The elements of $A$ are called {\em letters}. For $\ell \in \N$, the set of concatenations of $\ell$ letters is denoted by $A^{\ell}$ and $w = w_1 \ldots w_{\ell} \in A^{\ell}$ is a {\em word} of length $\ell$. The length of a word $w$ is denoted by $|w|$. We set $A^* = \bigcup_{n \in \N} A^{\ell}$ and by convention, $A^0 = \{ \varepsilon \}$ where $\varepsilon$ is the {\em empty word}. 

For a word $w = w_1  \ldots w_{\ell}$ and two integers $1 \leq i < j \leq \ell$, we write $w_{[i, j+1)} = w_{[i, j]} = w_i  \ldots w_j$. We say that $u$ {\em appears} or {\em occurs} in $w $ if there is an index $ 1 \leq i \leq |w|$ such that $u=w_{[i,i+|u|)}$ and we denote this by $u \sqsubseteq w$. The index $i$ is an {\em occurrence} of $u$ in $w$ and $|w|_u$ denotes the number of (possibly overleaped) occurrences of $u$ in $w$. We also write $\freq(u,w) = \frac{|w|_u}{|w|}$, the \emph{frequency of} $u$ \emph{in} $w$. 

Let $A^{\Z}$ be the set of two-sided sequences $(x_n)_{n \in \Z}$, where $x_n \in A$ for all $n \in \Z$. Like for finite words, for $x \in A^{\Z}$ and $- \infty < i < j < \infty$ we write $x_{[i,j]}= x_{[i,j+1)}$ for the finite word given by $x_ix_{i+1} \ldots x_j$. The set $A^{\Z}$ endowed with the product topology is a compact and metrizable space. The {\em shift map} $S\colon A^{\Z} \to A^{\Z}$ is the homeomorphism defined by $S((x_n)_{n \in \Z})= (x_{n+1})_{n \in \Z}$. Notice that, the collection of {\em cylinder sets} $\{ S^j[w] \colon w \in A^*, j \in \Z \}$ where $[w] = \{ x \in A^{\Z} \colon x_{[0, |w|) } = w\} $, is a basis of clopen subsets for the topology of $A^{\Z}$.  

A {\em subshift} is a topological dynamical system $(X,S)$, where $X$ is a closed and $S$-invariant subset of $A^{\Z}$. In this case the topology is also given by cylinder sets, denoted $[w]_X = [w] \cap X$, but when there is no ambiguity we just write $[w]$. Given an element $x \in X$, the \emph{language} $\cL(x)$ is the set of all words appearing in $x$ and $\cL(X) = \bigcup_{x \in X} \cL(x)$. Notice that $[w]_X \neq \emptyset$ if and only if $w \in \cL(X)$. Also, $(X,S)$ is minimal if and only if $\cL(X)=\cL(x)$ for all $x \in X$.

  Let $A$ and $B$ be finite alphabets and $\sigma\colon A^* \to B^*$ be a \emph{morphism} for the concatenation, that is $\sigma(uw) = \sigma(u)\sigma(w)$ for all $u,w \in A^*$. A morphism $\sigma\colon A^* \to B^*$ is completely determined by the values of $\sigma(a)$ for every letter $a \in A$. We only consider \emph{non-erasing} morphisms, that is $\sigma(a) \neq \varepsilon$ for every $a \in A$, where $\varepsilon$ is the empty word in $B^*$. A morphism $\sigma \colon A^* \to A^*$ is called a \emph{substitution} if for every $a \in A$, $\displaystyle \lim_{n \to \infty} |\sigma^n(a)| = \infty$.

A \emph{directive sequence} $\boldsymbol \sigma = (\sigma_n\colon A^*_{n+1} \to A^*_n )_{n \in \N}$ is a sequence of (non-erasing) morphisms.  Given a directive sequence $\boldsymbol \sigma$ and $n \in \N$, define $\cL^{(n)}(\boldsymbol \sigma)$, the \emph{language of level} $n$ \emph{associated to} $\boldsymbol \sigma $ by
\begin{equation*}
    \cL^{(n)}(\boldsymbol \sigma) = \{ w \in A_n^* : w \sqsubseteq \sigma_{[n,N)}(a) \text{ for some } a \in A_N \text{ and } N>n \}
\end{equation*}
where $\sigma_{[n,N)} = \sigma_n \circ \sigma_{n+1} \circ \ldots \circ \sigma_{N-1}$. For $n \in \N$, we define $X_{\boldsymbol \sigma}^{(n)}$, the $n$-\emph{th level subshift generated by} $\boldsymbol \sigma$, as the set of elements $x \in A_n^{\Z}$ such that $\cL(x) \subseteq \cL^{(n)}(\boldsymbol \sigma)$. For the special case $n=0$, we write $X_{\boldsymbol \sigma}$ instead of $X_{\boldsymbol \sigma}^{(0)}$ and we call it the $\cS$-\emph{adic subshift} generated by $\boldsymbol \sigma$.

A morphism $\sigma \colon A^* \to B^*$ has a \emph{composition matrix} $M(\sigma) \in \N^{B \times A} $ given by $M(\sigma)_{b,a} = |\sigma(a)|_b$ for all $b \in B$ and $a \in A$. If $\tau \colon B^* \to C^*$ is another morphism, then $M(\tau \circ \sigma) = M (\tau) M(\sigma)$. Therefore, for a substitution, $\sigma\colon A^* \to A^*$, $M(\sigma^2) = M(\sigma)^2$. 
We say that $\boldsymbol \sigma$ is {\em primitive} if for every $n \in \N$ there exists $k \geq 1$ such that the matrix $M (\sigma_{[n,n+k]}) = M(\sigma_n)M(\sigma_{n+1}) \cdots M( \sigma_{n+k})$ has only positive entries. When $\boldsymbol \sigma$ is primitive, then for every $n \in \N$ $(X_{\boldsymbol \sigma}^{(n)},S)$ is minimal and  $\cL(X^{(n)}_{\boldsymbol \sigma}) = \cL^{(n)}(\boldsymbol \sigma)$.

When $\boldsymbol \sigma$ is the constant directive sequence $\sigma_n = \sigma$ for all $n \in \N$, where $\sigma \colon A^* \to A^*$ is a substitution, then $X_{\boldsymbol \sigma}$ is denoted $X_{\sigma}$ and it is called \emph{substitution subshift}. Similarly $\cL(\boldsymbol \sigma)$ is denoted $\cL(\sigma)$. Also if in that context $\boldsymbol \sigma$ is primitive, we say that the substitution $\sigma$ itself is primitive, which is equivalent to saying that the composition matrix $M(\sigma)$ is primitive. We also say that the substitution $\sigma$ is positive if $M(\sigma)$ only have positive entries. By definition, every positive substitution is also primitive.

A morphism $\sigma\colon A^* \to B^*$ has constant length if there exists a number $\ell \geq 1$ such that $|\sigma(a)| = \ell$ for all $a \in A$. In this case, we write $| \sigma| = \ell$. More generally, a directive sequence $\boldsymbol \sigma = (\sigma_n\colon A^*_{n+1} \to A^*_n)_{n \in \N}$ is of \emph{constant-length} if each morphism $\sigma_n$ is of constant length. Notice that we do not require that $|\sigma_n| = |\sigma_m|$ for distinct $n,m\in \N$.

We define the \emph{alphabet rank} $AR$ of $\boldsymbol \sigma = (\sigma_n\colon A^*_{n+1} \to A^*_n )_{n \in \N}$ as 
$\displaystyle AR(\boldsymbol \sigma) = \liminf_{n \to \infty} |A_n|$. 
Having finite alphabet rank has many consequences, for instance if $AR(\boldsymbol \sigma) < \infty$ then $X_{\boldsymbol \sigma}$ has zero topological entropy.

For a general subshift $(X, S)$, let $p_X \colon \N \to \N$ denote \emph{the word complexity function} of $X$ given by $p_X (n) = |\cL_n (X)|$ for all $n \in \N$. Here $\cL_n(X) = \{ w \in \cL(X) \colon |w|=n\}$. If $\displaystyle \liminf_{n \to \infty} \frac{p_X(n)}{n} = \infty$ we say that $X$ has \emph{superlinear complexity}. Otherwise we say $X$ has \emph{non-superlinear complexity}.

 We say that a primitive substitution $\tau \colon A^* \to A^*$ is \emph{right prolongable} (resp. \emph{left prolongable}) on $u \in A^*$ if $\tau(u)$ starts (resp. ends) with $u$. If, for every letter $a \in A$, $\tau \colon A^* \to A^*$ is left and right prolongable on $a$, then $\tau \colon A^* \to A^*$ is said to be \emph{prolongable}. A word $w=w_1 \ldots w_{\ell}\in \cA^*$ is \emph{complete} if $\ell \geq 2$ and $w_1 = w_{\ell}$. Notice that if a substitution $\tau \colon A^* \to A^*$ is primitive and prolongable, then $\tau(a)$ is a complete word for every $a \in A$. If $W$ is a set of words, then we denote
\begin{equation} \label{eq complete W}
    \cC W = \{w \in W \colon |w| \geq 2, w_1 = w_{|w|} \}.
\end{equation}
the set of complete words in $W$. In particular, for $k \geq2$, $\cC A^k$ is the set of complete words of length $k$ with letters in $A$, for example, $\cC\{a,b\}^3= \{aaa,aba,bab,bbb\}$.

Finally, when the alphabet has two letters $\cA= \{a,b\}$, the \emph{complement} of a word $w = w_1 \ldots w_{\ell} \in \cA^*$ denoted $\overline{w}$ is given by $\overline{w}_1 \ldots \overline{w}_{\ell}$ where $\overline{a}= b$ and $\overline{b}=a$. A morphism $\tau \colon \cA^* \to \cA^*$ is said to be a mirror morphism if $\tau(\overline{w}) = \overline{\tau(w)}$ (the name is taken from \cite[Chapter 8.2]{Queffelec1987} with a slight modification). 

\subsection{Invariant measures} \label{section invariant measures}

A \emph{measure preserving system} is a tuple $(X,\mathcal{X},\mu,T)$, where $(X,\mathcal{X},\mu)$ is a probability space and $T\colon X\to X$ is a measurable and measure preserving transformation. That is, $T^{-1}A\in\mathcal{X}$ and $\mu(T^{-1}A)=\mu(A)$ for all $A\in \cX$, and we say that $\mu$ is $T$\emph{-invariant}. An invariant measure $\mu$ is said to be {\em ergodic} if whenever $A \subseteq X$ is measurable and $\mu(A\Delta T^{-1}A)=0$, then $\mu(A)=0$ or $1$.   

 Given a topological dynamical system $(X,T)$, we denote $\cM(X,T)$ (resp. $\cE(X,T)$) the set of Borel $T$-invariant probability measures (resp. the set of ergodic probability measures). For any topological dynamical system, $\cE(X,T)$ is nonempty and when $\cE(X,T) = \{ \mu\}$ the system is said to be {\em uniquely ergodic}.

If $(X,S)$ is a subshift over an alphabet $A$, then any invariant measure $\mu \in \cM(X,S)$ is uniquely determined by the values of $\mu([w]_X)$ for $w \in \cL(X)$. Since $X \subset A^{\Z}$, $\mu \in \cM(X,S)$ can be extended to $A^{\Z}$ by $\Tilde{\mu}( B) = \mu ( B \cap X) $ for all $B \subset A^{\Z} $ measurable. In particular, $\Tilde{\mu}([w]) = \mu ([w]_{X})$ for all $w \in A^*$. We use this extension many times, making a slight abuse of notation and not distinguishing between $\mu$ and $\Tilde{\mu}$. Moreover, for $w \in A^*$, since there is no ambiguity with the value of the cylinder set we write $\mu(w)$ instead of $\mu([w])$. This can also be done when we deal with two alphabets $A \subset B$, every invariant measure $\mu$ in $A^{\Z}$ can be extended to an invariant measure in $B^{\Z}$, where in particular, $\mu(b) =0 $ for all $b \in B\backslash A$.  

A sequence of non-empty subsets of the integers, $\boldsymbol{\Phi}= (\Phi_n)_{n\in \N} $ is a F\o lner sequence if for all $t \in \Z$, $\displaystyle \lim_{n \to \infty} \frac{|\Phi_n \Delta (\Phi_n+t)|}{|\Phi_n |} = 0$. Let $(X,T)$ be a topological system and let $\mu$ be an invariant measur, an element $x \in X$ is said to be \emph{generic} along $\boldsymbol \Phi$ if for every continuous function $f \in C(X)$
\begin{equation*}
    \lim_{n \to \infty} \frac{1}{|\Phi_n| } \sum_{k \in \Phi_n} f(Tx) = \int_X f d\mu. 
\end{equation*}
Every point in a minimal system is generic for some F\o lner sequence $\boldsymbol \Phi$, more precisely
\begin{proposition} \label{prop furstenberg generic}\cite[Proposition 3.9]{Furstenbergbook:1981}
    Let $(X,T)$ be a minimal system and $\mu$ an ergodic measure. Then for every $x \in X$ there exists sequences $(m_n)_{n \in \N}, (m'_n)_{n \in \N} \subset \N$ such that $m_n < m'_n$ for every $n \in \N$  and $\displaystyle \lim_{n \to \infty} m'_n - m_n = \infty$ such that $x$ is generic along $\boldsymbol \Phi = (\{m_n , \ldots, m'_n\})_{n \in \N}$.
\end{proposition}
In particular, for an $\cS$-adic subshift with primitive directive sequence $\boldsymbol \sigma = (\sigma_n \colon A_{n+1}^* \to A_n^*)_{n \in \N}$, when the infinite word $\boldsymbol w = \displaystyle \lim_{n \to \infty} \sigma_0 \circ \sigma_1 \circ \cdots \circ \sigma_{n-1}(a_n)$ is well-defined then every invariant measure $\mu \in \cM(X_{\boldsymbol \sigma},S)$ is given by 
\begin{equation} \label{equation empiric measure}
    \mu(u) =  \lim_{n \to \infty} \frac{|\boldsymbol{w}_{[m_n,m'_n]} |_u }{m'_n-m_n +1} = \lim_{n \to \infty} \freq(u,\boldsymbol{w}_{[m_n,m'_n]})  \quad \forall u \in \cL(X_{\boldsymbol \sigma}),  
\end{equation}
for some $(m_n)_{n \in \N}, (m'_n)_{n \in \N} \subset \N$ as before. Notice that such infinite word $\boldsymbol w$ is well-defined for example when  $A_n = A$, $a_n = a$ and $\sigma_n \colon A^* \to A^*$ is prolongable, for all $n \in \N$, where $A$ and $a \in A$ are a fixed alphabet and letter respectively. Those are the condition for the construction of the system announced in \cref{main thrm}.

We remark that for a primitive substitution, $\sigma \colon A^* \to A^*$ the substitution subshift $(X_{\sigma},S)$ is uniquely ergodic and the invariant measure is given by any limit of the form \eqref{equation empiric measure}.

\subsection{Partial rigidity rate for $\cS$-adic subshifts}

Every $\cS$-adic subshift can be endowed with a natural sequence of Kakutani-Rokhlin partitions see for instance \cite[Lemma 6.3]{Berthe_Steiner_Thuswaldner_Recognizability_morphism:2019}, \cite[Chapter 6]{Durand_Perrin_Dimension_groups_dynamical_systems:2022} or \cite[section 5]{donoso_maass_radic2023partial}. To do this appropriately, one requires \emph{recognizability} of the directive sequence $\boldsymbol \sigma = (\sigma_n \colon A_{n+1}^* \to A_n^*)_{n \in \N} $, where we are using the term recognizable as defined in \cite{Berthe_Steiner_Thuswaldner_Recognizability_morphism:2019}. We do not define it here, but if every morphism $\sigma_n \colon A_{n+1}^* \to A_n^*$ is left-permutative, that is the first letter of $\sigma_n(a)$ is distinct from the first letter of $\sigma_n(a')$ for all $a \neq a'$ in $A_n$, then the directive sequence is recognizable. In this case we say that the directive sequence $\boldsymbol \sigma$ itself is left-permutative. If $\tau \colon A^* \to A^*$ is prolongable, then it is left-permutative. 

Once we use the Kakutani-Rokhlin partition structure, $X^{(n)}_{\boldsymbol \sigma}$ can be identified as the induced system in the $n$-th basis and for every invariant measure $\mu'$ in $X^{(n)}_{\boldsymbol \sigma}$, there is an invariant measure $\mu$ in $X_{\boldsymbol \sigma}$ such that $\mu'$ is the induced measure of $\mu$ in $X^{(n)}_{\boldsymbol \sigma}$. We write $ \mu' = \mu^{(n)}$ and this correspondence is one-to-one. This is a crucial fact for computing the partial rigidity rate for an $\cS$-adic subshift, for instance, if $\boldsymbol \sigma$ is a directive sequence of constant-length, $\delta_{\mu} = \delta_{\mu^{(n)}}$ for all $\mu \in \cE(X_{\boldsymbol \sigma}, S)$ and $n \geq 1$ (see \cref{theorem constant length delta mu}). Since the aim of this paper is building a specific example, we give a way to characterize $\mu^{(n)}$ for a more restricted family of $\cS$-adic subshift that allows us to carry out computations.

In what follows, we restrict the analysis to less general directive sequences $\boldsymbol \sigma$. To do so, from now on, $\cA$ always denotes the two letters alphabet $\{a,b\}$. Likewise, for $d \geq 2$, $\cA_i = \{a_i, b_i\}$ for $i \in \{0, \ldots, d-1\}$ and $ \Lambda_d= \bigcup_{i=0}^{d-1} \cA_{i}$.

We cite a simplified version of \cite[Theorem 4.9]{bezuglyi_karpel_kwiatkowski2019exact}, the original proposition is stated for Bratelli-Vershik transformations, but under recognizability, it can be stated for $\cS$-adic subshifts, see \cite[Theorem 6.5]{Berthe_Steiner_Thuswaldner_Recognizability_morphism:2019}.

\begin{lemma} \label{lemma BKK}
    Let $\boldsymbol \sigma = (\sigma_n \colon \Lambda_d^* \to \Lambda_d^*)_{n \geq 1} $ be a recognizable constant-length and primitive directive sequence, such that for all $i \in \{0, \ldots, d-1\}$,
        \begin{equation} \label{eqa} 
            \lim_{n \to \infty}\frac{1}{|\sigma_n|} \sum_{j \neq i } |\sigma_n(a_i)|_{a_j} + |\sigma_n(a_i)|_{b_j} + |\sigma_n(b_i)|_{a_j} + |\sigma_n(b_i)|_{b_j} = 0
        \end{equation}
         
        \begin{equation} \label{eqc} 
            \sum_{n \geq 1} \left( 1- \min_{c \in \cA_i} \frac{1}{|\sigma_n|} \left( |\sigma_n(c)|_{a_i} + |\sigma_n(c)|_{b_i} \right) \right) < \infty
        \end{equation}
        \begin{equation} \label{eqd} 
           \text{and } \quad \lim_{n \to \infty} \frac{1}{| \sigma_n|} \max_{c,c' \in \cA_i} \sum_{d \in \Lambda_d} | |\sigma_n(c)|_d - |\sigma_n(c')|_d | =0.
        \end{equation}
Then the system $(X_{\boldsymbol \sigma},S)$ has $d$ ergodic measures $\mu_0, \ldots, \mu_{d-1}$. 
    
    Moreover, for $N \in \N$ sufficiently large, the measures $\mu^{(n)}_i$ are characterized by $\mu^{(n)}_i(a_i) + \mu^{(n)}_i (b_i) = \max \{ \mu' (a_i)+ \mu'(b_i) \colon \nu \in \cM(X_{\boldsymbol \sigma}^{(n)},S) \}$ for all $n \geq N$. Also, for all $j \neq i$, $$ \lim_{n \to \infty} \mu_i^{(n)}(a_j) + \mu_i^{(n)}(b_j) = 0.$$
\end{lemma}

 Whenever $\boldsymbol \sigma = (\sigma_n \colon A_{n+1}^* \to A_n^*)_{n \in \N}$ is  a constant-length  directive sequence, we write $h^{(n)} = |\sigma_{[0,n)}|$ where we recall that $\sigma_{[0,n)} = \sigma_0 \circ \sigma_1 \circ \cdots \circ \sigma_{n-1}$.

\begin{theorem} \cite[Theorem 7.1]{donoso_maass_radic2023partial}  \label{theorem constant length delta mu}
Let $\boldsymbol \sigma = (\sigma_n \colon A_{n+1}^* \to A_n^*)_{n \in \N}$  be a recognizable, constant-length and primitive directive sequence. Let $\mu$ be an $S$-invariant ergodic measure on $X_{\boldsymbol \sigma}$. Then
\begin{equation} \label{eq Toeplitz delta mu}
\delta_{\mu} = \lim_{n \to \infty }  \sup_{k \geq 2} \left\{ \sum_{w \in \cC A^k_n}  \mu^{(n)} (w) \right\},
\end{equation}
where $\cC A^k_n$ is defined in \eqref{eq complete W}. Moreover, if $(k_n)_{n \in \N}$ is a sequence of integers (posibly constant), with $k_n \geq 2$ for all $n \in \N$, such that 
\begin{equation} \label{eq constant length p rig rates}
\delta_{\mu} = \lim_{n \to \infty }   \left\{ \sum_{w \in \cC A_n^{k_n
}}  \mu^{(n)} (w) \right\},
\end{equation}
then the partial rigidity sequence is $((k_n-1) h^{(n)})_{n \in \N} $.
\end{theorem}

Another useful characterization of the invariant measures is given by explicit formulas between the invariant measures of $X_{\boldsymbol \sigma}^{(n)}$ and $X_{\boldsymbol \sigma}^{(n+1)}$. To do so we combine \cite[Proposition 1.1, Theorem 1.4]{bedaride_hilion_lusting_2023measureSadic} and \cite[Proposition 1.4]{bedaride_hilion_lusting_2022measureMonoid}. In the original statements one needs to normalize the measures to get a probability measure (see \cite[Proposition 1.3]{bedaride_hilion_lusting_2022measureMonoid}), but for constant length morphisms the normalization constant is precisely the length of the morphism. 
Before stating the lemma, for $\sigma \colon A^* \to B^*$, $w \in A^*$ and $u \in B^*$, we define $\lfloor \sigma(w) \rfloor_u$, the \emph{essential occurrence of} $u$ \emph{on} $\sigma(w)$, that is the number of times such that $u$ occurs on $w$ for which the first letter of $u$ occurs in the image of the first letter of $w$ under $\sigma$, and the last letter of $u$ occurs in the image of last letter of $w$ under $\sigma$. 

\begin{example*}
    Let $\sigma \colon \cA^* \to \cA^*$ given by $\sigma(a)=abab$ and $\sigma(b)=babb$. Then $\sigma(ab)=ababbabb$ and $|\sigma(ab)|_{abb} =2 $ but $\lfloor \sigma(ab) \rfloor_{abb}=1$.
\end{example*}

\begin{lemma} \label{lemma directive sequence measure formula} Let $\boldsymbol \sigma = (\sigma_n \colon A_{n+1}^* \to A_n^*)_{n \in \N}$ be a recognizable constant-length and primitive directive sequence and fix an arbitrary $n \in \N$. Then there is a bijection between $\cM (X_{\boldsymbol \sigma}^{(n)},S)$ and $\cM (X_{\boldsymbol \sigma}^{(n+1)},S)$. 

Moreover, for every invariant measure $\mu' \in \cM (X_{\boldsymbol \sigma}^{(n)},S)$, there is an invariant measure $\mu \in \cM (X_{\boldsymbol \sigma}^{(n+1)},S)$ such that for all words $u \in A_n^*$,

    \begin{equation} \label{eq formula1}
        \mu'(u) = \frac{1}{|\sigma_n|} \sum_{w \in W(u)} \lfloor \sigma_n(w) \rfloor_{u} \cdot \mu (w),
    \end{equation}
    where  $ \displaystyle W(u) = \left\{ w \colon |w| \leq \frac{|u|-2}{|\sigma_n|} + 2 \right\}$. Finally, if $\mu$ is ergodic, then $\mu'$ is also ergodic. 
\end{lemma}

\begin{corollary}
     Let $\boldsymbol \sigma = (\sigma_n \colon \Lambda_d^* \to \Lambda_d^*)_{n \in \N}  $ be a recognizable constant-length and primitive directive sequence that fulfills \eqref{eqa},\eqref{eqc} and \eqref{eqd} from \cref{lemma BKK}. Letting $\mu_0, \ldots, \mu_{d-1}$ denote the $d$ ergodic measures, then for $n\in \N$ sufficiently large

     \begin{equation} \label{eq formula2}
        \mu^{(n)}_i(u) = \frac{1}{|\sigma_n|} \sum_{w \in W(u)} \lfloor \sigma_n(w) \rfloor_{u} \cdot \mu^{(n+1)}_i (w) \quad \forall u \in \Lambda_d^*.
    \end{equation}
\end{corollary}

 \begin{proof}
     By the characterization given by \cref{lemma BKK} and using \eqref{eq formula1}
    \begin{align*}
        \mu^{(n)}_i(a_i) &+ \mu^{(n)}_i(b_i) = \max \{ \nu (a_i) + \nu (b_i) \colon \nu \in \cM(X_{\boldsymbol \sigma}^{(n)},S) \}
        \\ 
        &= \frac{1}{|\sigma_n|} \max\left\{ \sum_{c \in \Lambda_d} (| \sigma_n(c) |_{a_i} + | \sigma_n(c) |_{b_i})  \cdot \nu'(c) \mid \nu' \in \cM(X_{\boldsymbol \sigma}^{(n+1)},S)  \right\}.
    \end{align*}
    Using \eqref{eqc}, for big enough $n \in \N$, the invariant measure $\nu'$ that maximizes this equation has to be the invariant measure that maximize $\nu'(a_i)+\nu'(b_i)$ which is in fact $\mu^{(n+1)}_i$. 
     
 \end{proof}

 \begin{remark} \label{rmk letters to letters}
     When $\phi \colon A^* \to B^*$ is a letter to letter morphism, that is $|\phi(c)|=1$ for all $c \in A$, we have that $\phi$ induces a continuous map from $A^{\Z}$ to $B^{\Z}$ and that if $\mu$ is an invariant measure in $B^{\Z}$, then $ \mu' (w) = \displaystyle \sum_{u \in \phi^{-1}(w)} \mu (u)$ corresponds to the pushforward measure $\phi_* \mu$.
 \end{remark}

\section{The gluing technique and lower bound for the partial rigidity rates} \label{section gluing technique}

We recall that $\cA_i = \{a_i, b_i\}$ and $\Lambda_d = \bigcup_{i=0}^{d-1} \cA_i$. Let $\kappa \colon \Lambda^*_d \to \Lambda_d^*$ be the function that for every word of the form $ua_i$ (resp. $ub_i$) with $u\in \Lambda_d^*$, $\kappa(ua_i) = ua_{i+1}$ (resp. $\kappa(ub_i) = ub_{i+1}$) where the index $i \in \{0, \ldots,d-1\}$ is taken modulo $d$. For example, if $d=2$, $\kappa(a_0a_0) = a_0a_1 $, $\kappa(a_0b_0) = a_0b_1 $, $\kappa(a_0a_1) = a_0a_0 $ and $\kappa(a_0b_1) = a_0b_0 $. We highlight that the function $\kappa \colon \Lambda^*_d \to \Lambda_d^*$ is not a morphism.

For a finite collection of substitutions $\{ \tau_i \colon \cA_i^* \to \cA_i^* \mid i =0, \ldots, d-1\}$ we call the morphism $ \sigma = \Gamma( \tau_0, \ldots, \tau_{d-1}) \colon \Lambda_d^* \to \Lambda_d^*$ given by
\begin{align*}
    \sigma(a_i) &= \kappa(\tau_i(a_i)) \\
    \sigma(b_i) &= \kappa(\tau_i(b_i))
\end{align*}
for all $i \in \{0,\ldots,d-1\}$, the \emph{glued substitution} . This family of substitutions is the main ingredient for our construction. 

\begin{example*}
    Let $d=2$, $\tau_0 \colon \cA_0^* \to \cA_0^*$ and $\tau_1 \colon \cA_1^* \to \cA_1^*$ be the substitutions given by

    \begin{equation*}
        \begin{array}{cccc}
             \tau_0(a_0)&= a_0b_0b_0a_0 & \tau_0(b_0)&= b_0a_0a_0b_0,\\
             \tau_1(a_1)&= a_1b_1b_1b_1 & \tau_1(b_1)&= b_1a_1a_1a_1.
        \end{array}
    \end{equation*}
     Then $\sigma = \Gamma (\tau_0, \tau_1) \colon \Lambda_2^* \to \Lambda_2^*$ is given by
    \begin{equation*}
        \begin{array}{cccc}
             \sigma(a_0)&= a_0b_0b_0a_1 & \sigma(b_0)&= b_0a_0a_0b_1,\\
             \sigma(a_1)&= a_1b_1b_1b_0 & \sigma(b_1)&= b_1a_1a_1a_0
        \end{array}
    \end{equation*}
\end{example*}

\begin{lemma} \label{prop glued morphism}
    Let $\tau_i \colon \cA_i^* \to  \cA_i^*$ for $i = 0, \ldots d-1$ be a collection of positive and prolongable substitutions. Let $\boldsymbol \sigma = (\sigma_n \colon \Lambda_d \to \Lambda_d)_{n \in \N}$ be the directive sequence for which $\sigma_n = \Gamma (\tau^{n+1}_0, \ldots, \tau^{n+1}_{d-1})$, that is 
    \begin{align*}
        \sigma_n(a_i) &= \kappa(\tau_i^{n+1}(a_i)) \\
        \sigma_n(b_i) &= \kappa(\tau_i^{n+1}(b_i))
    \end{align*}
    for all $i \in \{0, \ldots, d-1\}$.  Then $\boldsymbol \sigma$ is primitive and left-permutative.

\end{lemma}

\begin{proof}
    Firstly, $\tau_0, \ldots, \tau_{d-1}$ are prolongable, in particular they are left-permutative and $\min\{|\tau_i(a_i)|,|\tau_i(b_i)|\} \geq 2$ for all $i \in \{0,\ldots,d-1\}$. Since the function $\kappa \colon \Lambda^*_d \to \Lambda^*_d$ does not change the first letter and every $\tau_i$ is defined over a different alphabet, the left permutativity is preserved.

    Secondly, $M(\sigma_n)_{c,d} = M(\tau_i^{n+1})_{c,d} - \1_{c=d}$ if $c,d$ are in the same alphabet $\cA_i$, $M(\sigma_n)_{a_{i+1},a_i} = M(\sigma_n)_{b_{i+1},b_i} =1$ and $M(\sigma_n)_{c,d} = 0$ otherwise. Notice that by positivity and prolongability, the sub-blocks $(M(\sigma_n)_{c,d})_{c,d \in \cA_i}$  are positive and therefore, for every $n \in \N$, $M(\sigma_{[n,n+d)})$ only has positive entries. 
    \end{proof}

\begin{theorem} \label{thrm gluing technique}
    Let $\tau_i \colon \cA_i^* \to  \cA_i^*$ for $i = 0, \ldots, d-1$ be a collection of positive and prolongable substitutions. Suppose that every substitution $\tau_i$ has constant length for the same length. Let $\boldsymbol \sigma = (\sigma_n \colon \Lambda_d \to \Lambda_d)_{n \in \N}$ be the directive sequence of glued substitutions $\sigma_n = \Gamma (\tau^{n+1}_0, \ldots, \tau^{n+1}_{d-1})$.
    Then the $\cS$-adic subshift $(X_{\boldsymbol \sigma},S)$ is minimal and has $d$ ergodic measures $\mu_0, \ldots, \mu_{d-1}$ such that for every $i \in \{0,\ldots,d-1\}$
    \begin{align} \label{eq limit}
       \lim_{n \to \infty} \mu^{(n)}_i(w) = \nu_i(w) \quad \text{ for all } w \in \cA_i^*  
    \end{align}
where $\nu_i$ is the unique invariant measure of the substitution subshift given by $\tau_i$. 
\end{theorem}

\begin{remark*}
    From \eqref{eq limit}, we get that $\displaystyle \lim_{n \to \infty} \mu^{(n)}_i(a_i) + \mu_i^{(n)}(b_i) = 1$ and therefore \\ $\displaystyle \lim_{n \to \infty} \mu^{(n)}_i(w) =0$ for all $w \not \in \cA_i^*$.
\end{remark*}

Before proving the theorem, we want to emphasize that this gluing technique can be easily generalized. Indeed, many of the hypothesis are not necessary but we include them to simplify notation and computations. 
For instance, restricting the analysis to substitutions defined over two letter alphabets is arbitrary. Also, the function $\kappa \colon \Lambda^*_d \to \Lambda_d^*$ could change more than one letter at the end of words. Furthermore, with an appropriated control of the growth, the number of letters replaced could even increase with the levels.

 One fact that seems critical for the conclusion of \cref{thrm gluing technique} is that $\boldsymbol \sigma$ is a constant-length directive sequence and that $\frac{1}{|\sigma_n|}M(\sigma_n)_{c,d}$ for two letters $c$ and $d$ in distinct alphabets $\cA_i$, $\cA_j$ goes to zero when $n$ goes to infinity.

\begin{proof}

By \cref{prop glued morphism}, $(X_{\boldsymbol \sigma},S)$ is minimal. Let $|\tau_i|= \ell$, which is well defined because the substitutions $\tau_0, \ldots, \tau_{d-1}$ all have the same length. Then, for every $n \in \N$, $\sigma_n = \Gamma(\tau_0^{n+1},\ldots, \tau_{d-1}^{n+1})$ has constant length $\ell^{n+1}$.  

We need to prove that $(X_{\boldsymbol \sigma},S)$ has $d$ ergodic measures, and so we check the hypotheses of \cref{lemma BKK}, 
    \begin{align*} 
            &\lim_{n \to \infty}\frac{1}{|\sigma_n|} \sum_{j \neq i } |\sigma_n(a_i)|_{a_j} + |\sigma_n(a_i)|_{b_j} + |\sigma_n(b_i)|_{a_j} + |\sigma_n(b_i)|_{b_j} \\
            &=  \lim_{n \to \infty}\frac{1}{\ell^{n+1}} (|\sigma_n(a_i)|_{a_{i+1}} + |\sigma_n(b_i)|_{b_{i+1}}) = \lim_{n \to \infty}\frac{2}{\ell^{n+1}} = 0. 
        \end{align*}
         This verifies \eqref{eqa}. Similarly for \eqref{eqc},
        \begin{equation*}
            \sum_{n \geq 1} \left( 1- \frac{1}{\ell^{n+1}} (|\sigma_n(a_i)|_{a_i} + |\sigma_n(a_i)|_{b_i})  \right) = \sum_{n \geq 1} \left( 1- \frac{\ell^{n+1}-1}{\ell^{n+1}} \right)  < \infty.
        \end{equation*}
        
        For \eqref{eqd}, notice that $|\sigma_n(a_i)|_{a_i} = |\tau_{i}^{n+1}(a_i)|_{a_i} -1$, therefore $\frac{1}{\ell^{n+1}} |\sigma_n(a_i)|_{a_i} = \freq (a_i, \tau^{n+1}(a_i)) - \frac{1}{\ell^{n+1}}$. Similarly for $|\sigma_n(a_i)|_{b_i}, |\sigma_n(b_i)|_{a_i}$ and $|\sigma_n(b_i)|_{b_i}$. Therefore
        \begin{align*}
            &\lim_{n \to \infty} \frac{1}{\ell^{n+1}} ||\sigma_n(a_i)|_{a_i} - |\sigma_n(b_i)|_{a_i} |    
            \\ =& \lim_{n \to \infty} |\freq(a_i, \tau_i^{n+1}(a_i)) - \freq(a_i, \tau_i^{n+1} (b_i)) |  = \nu_i(a_i) - \nu_i(a_i) =0.
        \end{align*}
        Likewise $\displaystyle \lim_{n \to \infty} \frac{1}{\ell^{n+1}} ||\sigma_n(a_i)|_{b_i} - |\sigma_n(b_i)|_{b_i} | = \nu_i(b_i) - \nu_i(b_i) = 0$. 

        Thus, by \cref{lemma BKK}, there are $d$ ergodic measures, $\mu_0, \ldots, \mu_{d-1}$ which are characterize by 
        \begin{equation} \label{eq measure charact}
            \mu^{(n)}_i(a_i) + \mu^{(n)}_i (b_i) = \max \{ \mu' (a_i)+ \mu'(b_i) \colon \mu' \in \cM(X_{\boldsymbol \sigma}^{(n)},S) \}
        \end{equation} 
        for sufficiently large $n \in \N$. The invariant measure that reaches the maximum in \eqref{eq measure charact} can be characterize as a limit like in \eqref{equation empiric measure}. Indeed, fix $n \in \N$ sufficiently large, $i \in \{0, \ldots, d-1\}$ and define the infinite one-sided word $\displaystyle \boldsymbol w^{(n)} = \lim_{k \to \infty} \sigma_{[n,n+k]} (a_i) = \lim_{k \to \infty} (\sigma_n \circ \cdots \circ \sigma_{n+k}) (a_i)$ and the number $N_k^{(n)}= |\sigma_{[n,n+k]} (a_i)|$ for every $k \in \N$. Let $\mu_n \in \cM(X_{\boldsymbol\sigma},S)$ be the measure given by
\begin{equation*} \label{eq de mu_n}
    \mu_n(u) =  \lim_{k \to \infty} \frac{1}{N^{(n)}_k} \left|\boldsymbol{w}^{(n)}_{[1,N^{(n)}_k]} \right|_u = \lim_{k \to \infty} \freq(u, \sigma_{[n,n+k]}(a_i)) 
\end{equation*}
for all $u \in \Lambda_d^*$. Notice that for any other F\o lner sequence of the form $(\{m_k, m_k+1, \ldots, m'_k\})_{k \in \N}$, $\displaystyle \lim_{k \to \infty} \frac{1}{m'_k-m_k} \left(  \left|\boldsymbol{w}^{(n)}_{[m_k,m'_k)} \right|_{a_i} + \left|\boldsymbol{w}^{(n)}_{[m_k,m'_k)} \right|_{b_i} \right) \leq \mu_n(a_i) + \mu_n(b_i)$. Thus, if $\mu'$ is given by $\displaystyle \mu'(u) = \lim_{k \to \infty} \frac{1}{m'_k-m_k}  \left|\boldsymbol{w}^{(n)}_{[m_k,m'_k)} \right|_{u} $ we get that $\mu'(a_i) + \mu'(b_i) \leq \mu_n(a_i) + \mu_n(b_i)$ and since every invariant measure $\mu' \in \cM(X_{\boldsymbol \sigma}^{(n)},S)$ has this form, $\mu_n = \mu_i^{(n)}$ by \eqref{eq measure charact}.

 To prove \eqref{eq limit}, fix $w \in \cA_i^*$ and $n \in \N$ large enough, then
        \begin{align}
            \mu_i^{(n)}(w) &=  \lim_{k \to \infty} \frac{|\sigma_{[n,n+k]}(a_i)|_w}{|\sigma_{[n,n+k]}(a_i)|} = \lim_{k \to \infty} \frac{|\sigma_{[n,n+k)} \circ \kappa (\tau_i^{n+k+1}(a_i))|_w}{|\sigma_{[n,n+k]}(a_i)|} \notag \\
            &\geq \lim_{k \to \infty} \frac{1}{|\sigma_{[n,n+k]}(a_i)|} \left( |\sigma_{[n,n+k)}(\tau_i^{n+k+1}(a_i))|_w - 1 + |\sigma_{[n,n+k)} (a_{i+1})|_w \right) \notag \\
            &\geq \lim_{k \to \infty} \frac{|\sigma_{[n,n+k)}(\tau_i^{n+k+1}(a_i))|_w }{|\sigma_{[n,n+k]}(a_i)|}, \label{ineq freq}
        \end{align}
where in the last inequality we use that $|\sigma_{[n,n+k]}| = \ell^{n} \cdot \ell^{n+1}\cdots  \ell^{n+k+1}$ and therefore $\frac{|\sigma_{[n,n+k)}|}{|\sigma_{[n,n+k]}|} = \frac{1}{\ell^{n+k+1}} \xrightarrow{k \to \infty} 0$. Notice that 
    \begin{align*}
        |\sigma_{[n,n+k)}(\tau_i^{n+k+1}(a_i))|_w &\geq |\sigma_{[n,n+k)}(a_i)|_w |\tau_i^{n+k+1}(a_i)|_{a_i} \\&+ |\sigma_{[n,n+k)}(b_i)|_w |\tau_i^{n+k+1}(a_i)|_{b_i}
    \end{align*}
    and since $|\tau_i^{n+k+1}(a_i)|_{a_i} + |\tau_i^{n+k+1}(a_i)|_{b_i} = \ell^{n+k+1}$ there exists $\lambda \in (0,1)$ such that

   \begin{equation*} 
        |\sigma_{[n,n+k)}(\tau_i^{n+k+1}(a_i))|_w \geq \ell^{n+k+1} \left( \lambda |\sigma_{[n,n+k)}(a_i)|_w + (1-\lambda) |\sigma_{[n,n+k)}(b_i)|_w \right).
    \end{equation*}
    
    Combining the previous inequality with \eqref{ineq freq} and supposing, without lost of generality, that $\displaystyle|\sigma_{[n,n+k)}(a_i)|_w = \min \{ |\sigma_{[n,n+k)}(a_i)|_w, |\sigma_{[n,n+k)}(b_i)|_w\}$, we get that 
    $$
            \mu_i^{(n)} (w) \geq \lim_{k \to \infty} \frac{ \ell^{n+k+1}}{|\sigma_{[n,n+k]}(a_i)|}   |\sigma_{[n,n+k)}(a_i)|_w. $$ 
            Now inductively
        \begin{align*}
           \mu_i^{(n)}(w) &\geq \lim_{k \to \infty} \frac{\ell^{n+2}   \ell^{n+3} \cdots  \ell^{n+k+1}} {|\sigma_{[n,n+k]}(a_i)|} |\tau_i^{n+1}(a_i)|_w   = \frac{ |\tau_i^{n+1}(a_i)|_w }{\ell^{n+1}},
        \end{align*}
        where in the last equality we use again that $|\sigma_{[n,n+k]}| = \ell^{n} \cdot \ell^{n+1}\cdots  \ell^{n+k+1}$. We conclude that $ \displaystyle \mu_i^{(n)}(w) \geq  \freq (w, \tau_i^{n+1}(a_i) )$, and then taking $n \to \infty$,
\begin{equation} \label{ineq final}
     \lim_{n \to \infty} \mu_i^{(n)}(w) \geq \lim_{n \to \infty}   \freq (w, \tau_i^n(a_i)) = \nu_i(w). 
\end{equation}

Since $w \in \cA_i^*$ was arbitrary \eqref{ineq final} holds for every word with letters in $\cA_i$. In particular, for every $k \geq 1$, $\displaystyle 1 = \sum_{u \in \cA_i^k} \nu_i(u) \leq \lim_{n \to\infty} \sum_{u \in \cA_i^k}  \mu_i^{(n)}(u) \leq 1$ which implies that the inequality in \eqref{ineq final} is an equality for every word $w \in \cA_i^*$. 
\end{proof}

In what follows every system $(X_{\boldsymbol \sigma}, S)$ and family of substitutions $\tau_i \colon \cA^*_i \to \cA^*_i$ for $i = 0, \ldots,d-1$ satisfy the assumption of \cref{thrm gluing technique}.

\begin{corollary}  $(X_{\boldsymbol \sigma},S)$ has non-superlinear complexity.
\end{corollary}

\begin{proof}
    This is direct from \cite[Corollary 6.7]{Donoso_Durand_Maass_Petite_interplay_finite_rank_Sadic:2021} where $\cS$-adic subshifts with finite alphabet rank and constant-length primitive  directive sequences have non-superlinear complexity.
\end{proof}

\begin{corollary} \label{cor delta smaler}
    If $\mu_0, \ldots, \mu_{d-1}$ are the ergodic measures of $(X_{\boldsymbol \sigma},S)$, then  
    \begin{equation}  \label{eq lower bound delta}
        \delta_{\nu_i} \leq \delta_{\mu_i}
    \end{equation}
for all $i \in \{0,\ldots,d-1\}$, where each $\nu_i$ is the unique invariant measure of $X_{\tau_i}$.

\end{corollary}

\begin{proof}
    By \cref{theorem constant length delta mu} equation \eqref{eq constant length p rig rates}, there exists a sequence of $(k_t)_{t \in \N}$ such that
\begin{equation*}
    \delta_{\nu_i} = \lim_{t \to \infty} \sum_{w \in \cC \cA_i^{k_t}}  \nu_i (w) 
\end{equation*}
and by \eqref{eq limit} for every $t \in \N$, there exists $n_t$ such that
\begin{equation*}
    \sum_{w \in \cC \cA_i^{k_t}}  \mu_i^{(n)} (w)  \geq  \sum_{w \in \cC \cA_i^{k_t}}   \nu_i (w)  - \frac{1}{t} \quad \text{ for all } n \geq n_t.
\end{equation*}
Taking limits we have,
\begin{equation*} 
    \delta_{\mu_i} \geq \lim_{t \to \infty} \left( \sum_{w \in \cC \cA_i^{k_t}}  \nu_i (w)  - \frac{1}{t} \right) = \delta_{\nu_i}. \qedhere
\end{equation*} 
\end{proof}
We finish this section with a case where the lower bound in \eqref{eq lower bound delta} is trivially achieved. For that, when we define a substitution $\tau \colon \cA^* \to \cA^*$ we abuse notation and write $\tau \colon \cA_i^* \to \cA_i^*$, by replacing the letters $a$ and $b$ by $a_i$ and $b_i$ respectively. Using that abuse of notation for $i \neq j$, we say that $\tau \colon \cA_i^* \to \cA_i^*$ and $\tau \colon \cA_j^* \to \cA_j^*$  are the \emph{same substitution} even though they are defined over different alphabets. We write $\Gamma(\tau,d) \colon \Lambda_d^* \to \Lambda_d^*$ when we are gluing $d$ times the same substitution. In the next corollary we prove that if we glue the same substitutions then we achieve the bound.

\begin{corollary} \label{cor one substitution}
        Let $\tau \colon \cA^* \to  \cA^*$  be a positive, prolongable and constant length substitution. Let $\boldsymbol \sigma = (\sigma_n \colon \Lambda_d \to \Lambda_d)_{n \in \N}$ be the directive sequence of glued substitutions $\sigma_n = \Gamma (\tau^{n+1},d)$.
    Then $(X_{\boldsymbol \sigma},S)$ has $d$ ergodic measures with the same partial rigidity rate $\delta_{\nu}$, where $\nu$ denotes the unique invariant measure of the substitution subshift $(X_{\tau},S)$.
\end{corollary}

\begin{proof}
    The letter-to-letter morphism $\phi \colon \Lambda_d^* \to \cA^*$ given by $a_i \mapsto a$ and $b_i \mapsto b$ for all $i=0,\ldots,d-1$ induce a factor map from $X_{\boldsymbol \sigma}$ to $X_{\tau}$ and therefore $\delta_{\mu} \leq \delta_{\nu}$ for all $\mu \in \cE(X_{\boldsymbol \sigma}, S)$ (see \cite[Proposition 1.13]{King_joining-rank_finite_mixing:1988}). The opposite inequality is given by \cref{cor delta smaler}.
\end{proof}

\section{Computation of the partial rigidity rates} \label{section computation partial rigidity}

\subsection{Decomposition of the directive sequence}

We maintain the notation, using $\cA_i = \{a_i,b_i \} $ and  $\Lambda_d = \bigcup_{i=0}^{d-1} \cA_i$ and we also fix $\cA_i' = \{a_i', b_i'\}$, $\Lambda_d' = \bigcup_{i=0}^{d-1} \cA_i \cup \cA_i'$. In this section, $\tau_i \colon \cA^*_i \to \cA_i^*$ for $i = 0, \ldots, d-1$ is a collection of mirror substitutions satisfying the hypothesis of \cref{thrm gluing technique}, $\ell = |\tau_i|$ and $\boldsymbol \sigma = ( \Gamma(\tau_0^{n+1}, \ldots, \tau_{d-1}^{n+1}))_{n \in \N}$, that is
 \begin{align*}
        \sigma_n(a_i) &= \kappa(\tau_i^{n+1}(a_i)) \\
        \sigma_n(b_i) &= \kappa(\tau_i^{n+1}(b_i))
    \end{align*}
for all $i \in \{0, \ldots,d-1\}$. We also write $\cE$ instead of $\cE(X_{\boldsymbol \sigma}, S)= \{\mu_0, \ldots, \mu_{d-1}\}$ for the set of ergodic measures. 

\begin{proposition}
    The directive sequence $\boldsymbol \sigma$ can be decomposed using $3$ morphisms in the following way: for every $n \in \N$, $\sigma_n = \phi \circ \rho^{n} \circ \psi$ where
\begin{align*}
    \psi \colon \Lambda_d^* \to (\Lambda_d')^* & \quad a_i \mapsto u_i a_{i+1}' \\
    & \quad b_i \mapsto v_i b_{i+1}'\\
    \\
    \rho \colon (\Lambda_d')^* \to (\Lambda_d')^* & \quad a_i \mapsto \tau_i(a_i) \quad  a_i' \mapsto  u_{i-1} a_i' \\
    & \quad b_i \mapsto \tau_i (b_i) \quad b_i' \mapsto  v_{i-1} b_i' \\
    \\
    \phi \colon (\Lambda_d')^* \to \Lambda_d^* & \quad a_i \mapsto a_i \quad  a_i' \mapsto a_{i} \\
    & \quad b_i \mapsto b_i \quad b_i' \mapsto b_{i}.
\end{align*}
with $u_i = \tau_i(a_i)_{[1,\ell)}$ and $v_i = \tau_i(b_i)_{[1,\ell)}$ and the index $i$ is taken modulo $d$. 
\end{proposition}

\begin{proof}
    Fix $i \in \{0,\ldots,d-1\}$. 
    Consider first that for every $n \geq 1$, $\rho^n(a_{i+1}') = \rho^{n-1}(u_i)\rho^{n-1}(a_{i+1}')= \tau_i^{n-1}(u_i)\rho^{n-1}(a_{i+1}')$, therefore by induction $$\rho^n(a_{i+1}') = \tau_i^{n-1}(u_i)\tau_i^{n-2}(u_{i}) \cdots \tau_i(u_i)u_ia_{i+1}' .$$
    
    Since, by assumption, the last letter of $\tau_i(a_i)$ is $a_i$, one gets that $\tau_i^{n-1}(u_i)\tau_i^{n-2}(u_{i}) $ $ \cdots \tau_i(u_i)u_i = \tau^{n}(a_i)_{[1,\ell^n)}$ and then $\rho^n(a_{i+1}') =  \tau^{n}(a_i)_{[1,\ell^n)} a_{i+1}'$. Also, we notice that $\psi(a_i) = \rho(a_{i+1}')$ and therefore $\rho^n \circ \psi(a_i) = \rho^{n+1}(a_{i+1}') = \tau^{n+1}(a_i)_{[1,\ell^{n+1})} a_{i+1}'  $. 
    
    Finally, $\displaystyle \phi \circ \rho^n \circ \psi(a_i) = \phi( \tau^{n+1}(a_i)_{[1,\ell^{n+1})}) \phi(a_{i+1}') = \tau^{n+1}(a_i)_{[1,\ell^{n+1})} a_{i+1} = \kappa(\tau^{n+1}(a_i))= \sigma_n(a_i)  .$ 
    We conclude noticing that the same proof works for $b_i$. 
\end{proof}

With this decomposition, we make an abuse of notation and define a directive sequence $\boldsymbol \sigma '$ over an index $Q$ different from $\N$. 

Set $\displaystyle Q = \{0\} \cup \bigcup_{n \geq 1} \left\{ n + \frac{m}{n+2}: m = 0, \ldots, n+1 \right\} $ we define the directive sequence $\boldsymbol \sigma' $ indexed by $Q$ given by 

\begin{equation*}
\sigma'_q =
    \begin{cases}
        \begin{array}{cc}
            \phi & \text{ if } q=n  \\
            \rho & \text{ if } q=n + m/(n+2) \text{ for } m=1, \ldots, n \\
            \psi & \text{ if } q=n + (n+1)/(n+2) 
        \end{array}
    \end{cases}
\end{equation*}
for all $n \geq 1$. We use this abuse of notation, in order to get $X^{(n)}_{\boldsymbol \sigma} = X^{(n)}_{\boldsymbol \sigma'}$ for every positive integer $n$, and therefore we maintain the notation for $\mu^{(n)}_i$. The advantage of decomposing the directive sequence is that every morphism in $\boldsymbol \sigma$ has constant length, either $\ell$ in the case of $\psi$ and $\rho$ or $1$ in the case of $\phi$. This simplifies the study of the complete words at each level.  Notice that, the morphisms $\phi$, $\rho$ and $\psi$ are not positive, otherwise the $\cS$-adic subshift would automatically be uniquely ergodic, see \cite{Durand2000}, which does not happen as we show in \cref{thrm gluing technique}.

\subsection{Recurrence formulas for complete words} The formulas in this section are analogous to those presented in \cite[Lemma 7.7]{donoso_maass_radic2023partial}, and aside from technicalities, the proofs are not so different.

We define four sets of words that are useful in what follows,
\begin{align}
C_k^i&= \{ w \in \Lambda_d^k \colon w_1,w_k \in \cA_i \cup \cA_{i+1}', w_1 = w_k\} \label{equation C}\\
    D_k^i&= \{ w \in (\Lambda_d')^k \colon w_1,w_k \in \cA_i \cup \cA_{i+1}', \eta(w_1) = \eta(w_k)\} \label{equation D}\\
    \overline{C}_k^i&= \{ w \in \Lambda_d^k \colon w_1,w_k \in \cA_i \cup \cA_{i+1}', w_1 = \overline{w_k} \} \\
    \overline{D}_k^i&= \{ w \in (\Lambda_d')^k \colon w_1,w_k \in \cA_i \cup \cA_{i+1}', \eta(w_1) = \overline{\eta(w_k)}\} \label{equation D bar}
\end{align}
where $\eta \colon \Lambda_{d}' \to \Lambda_{d}$ is a letter-to-letter function for which $a_i \mapsto a_i$, $b_i \mapsto b_i$, $a_{i+1}' \mapsto a_{i}$ and $b_{i+1}' \mapsto b_i$. For instance if $w \in D_k^i$ and $w_1 = a_i$ then $w_k \in \{a_i, a_{i+1}'\}$. 

To simplify the notation, we enumerate the index set $Q = \{q_m \colon m \in \N\}$ where $q_{m} < q_{m+1}$ for all $m \in \N$. We continue using the abuse of notation $\mu(w) = \mu([w])$ and for a set of words $W$, $\displaystyle \mu(W) = \mu \left(\bigcup_{w \in W} [w]\right)$. 

For $i \in \{0, \ldots, d-1\}$, fix the word $v= \tau_i(a_i)$ and we define $\delta_{j,j'}^{i} = \1_{v_j = v_{j'}}$ for $j, j' = \{1,\ldots, \ell\}$ where $\ell = |v|$. Notice that if one defines $\delta_{j,j'}^{i}$ with the word $\tau_i(b_i)$ instead of $\tau_i(a_i)$, by the mirror property, the value remains the same. Now, for $j \in \{ 1, \ldots, \ell\}$, we define
\begin{equation*}
    r_j^{i} = \sum^{j}_{j'=1} \delta_{\ell-j + j', j'}^i \quad \text{ and } \quad
    \Tilde{r}_j^{i} = \sum^{\ell-j}_{j'=1} \delta_{j', j+j'}^i.
\end{equation*}

\begin{lemma} \label{lemma complete rho}
    If $\boldsymbol \sigma' = (\sigma'_q)_{q \in Q}$ and $\mu \in \cE$, then  for every $n \in \N$, and every $q_m = n + \frac{m'}{n+2}$ for $m' \in \{1, \ldots, n\}$,
    \begin{align*}
       \ell \cdot \mu^{(q_m)} (D^i_{\ell k + j }) = &  r^i_j \cdot \mu^{(q_{m+1})} (D^i_{k+2}) + \Tilde{r}^i_j \cdot \mu^{(q_{m+1})} (D^i_{k+1}) \\
       &+ (j -r^i_j) \mu^{(q_{m+1})} (\overline{D}^i_{k+2}) + (\ell-j-\Tilde{r}^i_j) \mu^{(q_{m+1})} (\overline{D}^i_{k+1}) \\
        \\
        \ell \cdot \mu^{(q_m)} (\overline{D}^i_{\ell k + j }) = &  (j - r^i_j) \mu^{(q_{m+1})} (D^i_{k+2}) + (\ell-j- \Tilde{r}^i_j) \mu^{(q_{m+1})} (D^i_{k+1}) \\
       &+ r^i_j \cdot \mu^{(q_{m+1})} (\overline{D}^i_{k+2}) + \Tilde{r}^i_j \cdot  \mu^{(q_{m+1})} (\overline{D}^i_{k+1}) 
    \end{align*}
for $j \in \{1, \ldots, \ell\}$, where the set $D^i_k$ was defined in \eqref{equation D}.

\end{lemma}

\begin{proof}
    Notice that in this case $\sigma'_{q} = \rho $. 
    
    If $w \in \cL(X^{(q_m)}_{\boldsymbol{\sigma'}})$ for which $w_1 \in \cA_i \cup \cA_{i+1}'$, then $w \sqsubseteq \rho(u)$, where $u \in \cL(X^{(q_{m+1})}_{\boldsymbol{\sigma'}})$ and $u_1 \in \cA_i \cup \cA_{i+1}'$. This is equivalent to the condition $\eta(u_1) \in \cA_i$ . Since $\eta(\rho(a_i)) =\eta(\rho(a_{i+1}')) = \tau_i(a_i)$ and $\eta(\rho(b_i)) = \eta(\rho(b_{i+1}')) = \tau_i(b_i)$, for $u \in \cL(X^{(q_{m+1})}_{\boldsymbol{\sigma'}})$ satisfying $\eta(u_1) \in \cA_i$, we deduce that if $|u|=k+2$ with $\eta(u_1) = \eta(u_k)$, then
\begin{equation*}
   r^i_j = \sum_{j'=1}^j\1_{\eta(\rho(u_1)_{\ell -j -j'}) = \eta(\rho(u_{k+2})_{j'}) }
\end{equation*}
and when we consider $\eta(u_1) = \overline{\eta(u_{k+2})}$, $\displaystyle j - r^i_j = \sum_{j'=1}^j \1_{\eta(\rho(\overline{u}_1)_{\ell -j -j'}) = \eta(\rho(u_{k+2})_{j'}) }$. If $|u|=k+1$ with $\eta(u_1) = \eta(u_k)$
\begin{equation*}
    \Tilde{r}^i_j = \sum_{j'=1}^{\ell-j} \1_{\eta(\rho(u_1)_{j'}) = \eta(\rho(u_{k+1})_{j+j'}) } 
\end{equation*}
and when we consider $\eta(u_1) = \overline{\eta(u_{k+1})}$, $\displaystyle \ell - j - \Tilde{r}^i_j = \sum_{j'=1}^{\ell-j} \1_{\eta(\rho(\overline{u}_1)_{j'}) = \eta(\rho(u_{k+1})_{j+j'}) }$. 

Thus, the first equality of the lemma is a direct consequence of \eqref{eq formula2} and the second equality is completely analogous.

    \end{proof}

\begin{lemma} \label{lemma complete psi}
    If $\boldsymbol \sigma' = (\sigma'_q)_{q \in Q}$ and $\mu \in \cE$, then  for every $n \in \N$, let $q = n + \frac{n+1}{n+2}$, we get

\begin{align*}
       \ell \cdot \mu^{(q_m)} (D^i_{\ell k + j }) = &  r^i_j \cdot \mu^{(q_{m+1})} (C^i_{k+2}) + \Tilde{r}^i_j \cdot \mu^{(q_{m+1})} (C^i_{k+1}) \\
       &+ (j -r^i_j) \mu^{(q_{m+1})} (\overline{C}^i_{k+2}) + (\ell-j-\Tilde{r}^i_j) \mu^{(q_{m+1})} (\overline{C}^i_{k+1}) \\
        \\
        \ell \cdot \mu^{(q_m)} (\overline{D}^i_{\ell k + j }) = &  (j - r^i_j) \mu^{(q_{m+1})} (C^i_{k+2}) + (\ell-j- \Tilde{r}^i_j) \mu^{(q_{m+1})} (C^i_{k+1}) \\
       &+ r^i_j \cdot \mu^{(q_{m+1})} (\overline{C}^i_{k+2}) + \Tilde{r}^i_j \cdot  \mu^{(q_{m+1})} (\overline{C}^i_{k+1}) 
    \end{align*}
    
for $j \in \{1, \ldots, \ell\}$.

\end{lemma}

\begin{proof}
    Noting $\sigma'_{q_m} = \psi $ and that $\psi(a_i)=\rho(a_{i+1}')$ for all $i \in \{0, \ldots, d-1\}$, one can repeat the steps of \cref{lemma complete rho} proof and deduce the formula. 
\end{proof}

\begin{lemma} \label{lemma complete phi}
    If $\boldsymbol \sigma' = (\sigma'_q)_{q \in Q}$ and $\mu \in \cE$, then  for every $q_m = n \in \N$, 

    \begin{align}
        \mu^{(n)} (C^i_{k}) &\leq \mu^{(q_{m+1})} (D^i_{k})  + \frac{2}{\ell^{n+1}} \label{ineq C_k}\\
        \mu^{(n)} (\overline{C}^i_{k}) &\leq \mu^{(q_{m+1})} (\overline{D}^i_{k}) + \frac{2}{\ell^{n+1}} \label{ineq over C_k}
    \end{align}

\end{lemma}

\begin{proof}
    Notice that $\sigma'_{n} = \phi $ is letter-to-letter so by  \cref{rmk letters to letters}
    \begin{equation*}
        \mu^{(n)} (w) = \sum_{u \in \phi^{-1}(w)} \mu^{(q_{m+1})} (u).
    \end{equation*}
    
    The set $\phi^{-1}(C_k^i)$ is contained in $U \cup U'$ where $U$ is the set of complete words $u$ with length $k$ and first letter in $\cA_i$ and $U'$ is the set of words $u$ with length $k$ and first or last letter in $\cA_i'$. With that,
    \begin{align*}
        \mu^{(n)} (C_k^i) \leq& \mu^{(q_{m+1})} (U)
        +  \mu^{(q_{m+1})} (U') \\
        \leq & \mu^{(q_{m+1})}(D^i_k) + 2( \mu^{(q_{m+1})}(a_i') + \mu^{(q_{m+1})}(b_i')) 
        \leq \mu^{(q_{m+1})}(D^i_k) + \frac{2}{\ell^{n+1}}.
    \end{align*}
    where the last inequality uses that, by induction, $ \mu^{(q_{m+1})}(a_i') = \frac{1}{\ell^{n+1}} \mu^{(n+1)}(a_{i-1}) \leq \frac{1}{2 \ell^{n+1}}$. Likewise, $ \mu^{(q_{m+1})}(b_i')  \leq \frac{1}{2 \ell^{n+1}}$. 
    Inequality \eqref{ineq over C_k} uses the same reasoning.
    
\end{proof}

\subsection{Upper bounds}

Recall the definition of $C^i_k$, $D^i_k$, $\overline{C}^i_k$ and $\overline{D}^i_k$ given by the equations \eqref{equation C} to \eqref{equation D bar}. 

\begin{lemma} \label{lemma i constant length bound}
    For every $\mu \in \cE$ $n \in \N$ and $k \geq 2$, 

    \begin{equation} \label{ineq max all levels}
        \mu^{(n)} (C^i_{k}) \leq \max_{\substack{k' =2, \ldots, \ell \\ q \in Q, q\geq n} }  \{ \mu^{(q)} (D^i_{k'}) , \mu^{(q)} (\overline{D}^i_{k'})  \} + \frac{\ell }{\ell -1 }\frac{2}{\ell^{n+1}}.
    \end{equation}
\end{lemma}

\begin{remark*}
    Following what we discuss in \cref{section invariant measures} in the right hand side, if $q$ is an integer, $\mu^{(q)}$ is supported in $\Lambda_d^{\Z}$ and therefore it can be studied as a measure in $(\Lambda_d')^{\Z}$. In that context, $\mu^{(q)}(D^i_{k'}) =  \mu^{(q)}(C^i_{k'}) $ and $\mu^{(q)}(\overline{D}^i_{k'}) =  \mu^{(q)}(\overline{C}^i_{k'}) $, because $\mu^{(q)}(w) = 0$ whenever $w$ contains a letter in $\Lambda_d' \backslash \Lambda_d$.
\end{remark*}

\begin{proof}
    Combining Lemmas \ref{lemma complete rho} and  \ref{lemma complete psi} we deduce that for $q_m \in Q \backslash \N$, $\mu^{(q_m)} (D^i_{\ell k + j })$ and $\mu^{(q_m)} (\overline{D}^i_{\ell k + j })$ are convex combinations of $\mu^{(q_{m+1})} (D^i_{k + s })$ and $\mu^{(q_{m+1})} (\overline{D}^i_{k + s})$ for $s=1,2$. Therefore, if $q_m \in Q \backslash \N$
    
    \begin{equation*}
        \mu^{(q_m)} (D^i_{\ell k + j }) \leq \max_{s=1,2}\{ \mu^{(q_{m+1})} (D^i_{k + s }), \mu^{(q_{m+1})} (\overline{D}^i_{k + s})\}
    \end{equation*}
    and the same bound holds for $\mu^{(q_m)} (\overline{D}^i_{\ell k + j })$. Likewise, using \cref{lemma complete phi} for $q_m \in\N$, 
    \begin{align*}
        \mu^{(q_m)} (D^i_{k}) & \leq \mu^{(q_{m+1})} (D^i_{k }) + \frac{2}{\ell^{n+1}} \\
        \mu^{(q_m)} (\overline{D}^i_{k}) &\leq  \mu^{(q_{m+1})} (\overline{D}^i_{k }) + \frac{2}{\ell^{n+1}}
    \end{align*}
    Notice that for $2 \leq k \leq \ell$, the proposition is trivial. Thus, fix $k > \ell $, there exists an integer $k_1 \in \N$ and $m_1 \in \{1, \ldots, \ell\}$ such that $k = \ell \cdot k_1 + m_1 $. 

    Now, take $q_m = n \in \N$, then by the previous inequalities
    \begin{align*}
        \mu^{(n)} (C^i_{k}) & \leq \mu^{(q_{m+1})} (D^i_{k}) + \frac{2}{\ell^{n+1}}  \label{ineq first step}\\
        \mu^{(q_{m+1})} (D^i_{k}) &
        \leq \max_{s=1,2}\{ \mu^{(q_{m+2})} (D^i_{k_1 + s }), \mu^{(q_{m+2})} (\overline{D}^i_{k_1 + s})\} 
    \end{align*}

    If $k_1 \in \{1, \ldots, \ell -2\}$ we are done. If $k_1 = \ell -1$, we need to control the values indexed by $k_1+2 = \ell +1$, but for that we need to iterate the argument one more time. Otherwise, that is if $k_1 \geq \ell $, we can find $k_2 \geq 1$ and $m_2 \in \{1, \ldots, \ell\}$ such that $k_1 + 1 = \ell k_2 + m_2$ (similarly for $k_1 + 2 = \ell k_2 + m_2 +1$ or, if $m_2 = \ell$, $k_1 + 2 = \ell (k_2+1) + 1$). With that decomposition one can bound the right hand side of the second equality by $\displaystyle \max_{s = 1, 2, 3} \{ \mu^{(q_{m+3})} (D^i_{k_2 + s}), \mu^{(q_{m+3})} (\overline{D}^i_{k_2 + s}) \}$. 

    Consider the sequence, $(k_t)_{t \in \N}$ and $(m_t)_{t \geq 1}$ such that $k_t \geq 0$ and $m_t \in \{1,\ldots, \ell \}$ and are defined as follow, $k_0 = k$, $k_0 = \ell k_1 + m_1$ and inductively $k_t = \ell (k_{t+1} + t) + m_t $. Then eventually $k_t = 0$ for some $t \in \N$. With that, one can iterate the previous argument a finite amount of time and be able to express everything with only values $k' \in \{2, \ldots, \ell \}$. The only problem is when $n \leq \overline{n} = q_{m+t} \in \N$ in that case, we are force to add the term $ 2/ \ell^{\overline{n}+1}$. So we get 
    \begin{equation*}
        \mu^{(n)} (C^i_{k}) \leq \max_{\substack{k' =2, \ldots, \ell \\ q \in Q, n \leq q < N} }  \{ \mu^{(q)} (D^i_{k'}) , \mu^{(q)} (\overline{D}^i_{k'})  \} + \frac{2}{\ell^{n+1}} + \frac{2}{\ell^{n+2}} + \cdots + \frac{2}{\ell^{N}} 
    \end{equation*}
for some $N \geq n$, but that value is bounded by 
    $$\max_{\substack{k' =2, \ldots, \ell \\ q \in Q, q \geq n} }  \{ \mu^{(q)} (D^i_{k'}) , \mu^{(q)} (\overline{D}^i_{k'})  \} + \sum_{s \geq 1} \frac{2}{\ell^{n+s}}, $$
    which finish the proof. \vspace{-0.5em}    
\end{proof}

\begin{proposition} \label{thrm combination bound max}
    For every $i \in \{0, \ldots, d-1\}$, 

    \begin{equation*}
        \delta_{\mu_i} \leq \max_{k=2, \ldots, \ell  }  \left\{ \sum_{ w \in \cC \cA_i^k} \nu_i ( w) ,\sum_{w \in \overline{\cC} \cA_i^k} \nu_i (w)  \right\}
    \end{equation*}
    where the notation $\cC \cA_i^k$ is introduced in \eqref{eq complete W} and $\overline{\cC}\cA^k_i$ is the set of words $w \in \cA_i^*$ of length $k$ such that $w_1 = \overline{w}_k$ 
\end{proposition}
    
\begin{proof} First notice that, for every $(k_t)_{t \in \N}$ a possibly constant sequence of integers greatest or equal than $2$, 
\begin{align*}
    \lim_{t \to \infty} \sum_{w \in \cC \Lambda_d^{k_t}} \mu_i^{(t)} (w) &= \lim_{t \to \infty} \sum_{w \in \cC \Lambda_d^{k_t}, w_1 \in \cA_i} \mu_i^{(t)} (w) + \lim_{t \to \infty} \sum_{w \in \cC \Lambda_d^{k_t}, w_1 \not \in \cA_i} \mu_i^{(t)} (w)  \\
    &\leq \lim_{t \to \infty} \mu_i^{(t)} (C_{k_t}^i) + \lim_{t \to \infty} \sum_{c \in \Lambda_d \backslash \cA_i} \mu_i^{(t)} (c) = \lim_{t \to \infty} \mu_i^{(t)} (C_{k_t}^i) 
\end{align*}

Therefore, by \cref{theorem constant length delta mu} we get that there exists $(k_t)_{t \in \N}$ a possibly constant sequence of integers greatest or equal than $2$ such that 

\begin{align*}
    \delta_{\mu_i} &= \lim_{t \to \infty} \sum_{w \in \cC \Lambda_d^{k_t}}  \mu_i^{(t)} (w) \leq \lim_{t \to \infty} \mu_i^{(t)} (C_{k_t}^i)  \leq \lim_{t \to \infty}   \max_{\substack{k' =2, \ldots, \ell \\ q \in Q, q\geq t} }  \{ \mu^{(q)} (D^i_{k'}) , \mu^{(q)} (\overline{D}^i_{k'})  \}   
\end{align*}
where the last inequality is a consequence of \eqref{ineq max all levels}.

    Thus, we only have to control the values of $\mu^{(q)}(D^i_k)$ and $\mu^{(q)}(\overline{D}^i_k)$ for $k \in \{2, \ldots, \ell\}$ and big $q \in Q$. This is already controlled when $q$ is an integer because, \cref{thrm gluing technique} implies that for every $\epsilon>0$, there exists $N\geq 1$ such that for every $n \geq N$ and every word $w \in \cA^*_i$, with $|w|\leq \ell$,  $\mu_i^{(n)}(w) \leq \nu_i(w) + \varepsilon$ and $w \not \in \cA_i^*$, $\mu_i^{(n)}(w) \leq  \frac{\varepsilon}{2}$.

    Now, fix $q = n_1 + \frac{m'}{n_1 + 2} \not \in \N$ and $n_1 \geq N$ , notice that for $j \neq i$, $$\mu^{(q)}_i(D^j_k) \leq \sum_{c \in \cA_j \cup \cA_{j+1}'} \mu^{(q)}_i(c)    \leq  \mu_i^{(n_1 +1)}(a_j) + \mu_i^{(n_1 +1)}(a_j) \leq \varepsilon.$$
    
 If one repeats a proof similar to the one of \cref{thrm gluing technique} for the subshift $\eta(X_{\boldsymbol \sigma'}^{(q)})$, we get that for every $w \in \cA^*_i$, with $|w|\leq \ell$, $\eta_*\mu_i^{(q)}(w) \leq \nu_i(w) + \varepsilon$. Noting that, for $k' \leq \ell$, if $w \in D^i_{k'}$ then $\eta(w) \in \cC \cA_i^{k'}$ we deduce 

    \begin{equation*}
        \mu^{(q)}_i (D^i_{k'}) \leq \eta_* \mu^{(q)}_i (\cC \cA_i^{k'}) 
        \leq \sum_{u \in \cC \cA_i^{k'}}  (\nu_i (u) + \varepsilon) \leq 2^{k'} \varepsilon + \nu_i (\cC \cA_i^{k'}).
    \end{equation*}
     Similarly $\mu^{(q)}_i (\overline{D}^i_{k'})  \leq 2^{k'} \varepsilon + \nu_i (\overline{\cC} \cA_i^{k'})$. Therefore for every $\varepsilon >0$ there exists $N$, such that for every $n \geq N$
    \begin{equation*} 
        \max_{\substack{k' =2, \ldots, \ell \\ q \in Q, q\geq n} }  \{ \mu^{(q)} (C^i_{k'}) , \mu^{(q)} (\overline{C}^i_{k'})  \} \leq 2^{\ell} \varepsilon +  \max_{k=2, \ldots, \ell  }  \left\{\nu_i (\cC \cA_i^{k'}),\nu_i (\overline{\cC} \cA_i^{k'})  \right\} 
    \end{equation*}
    Thus taking limit $n \to \infty$ and $\varepsilon \to 0$ and  we conclude. 
\end{proof}

\subsection{System with multiple partial rigidity rates} We use the result of the last section of \cite{donoso_maass_radic2023partial}, for that fix $L \geq 6$ and let $\zeta_L \colon \cA^* \to \cA^*$ given by 
\begin{align*}
    a \mapsto a^Lb \\
    b \mapsto b^La.
\end{align*}
In particular $\zeta_L^2 $ is a prolongable and mirror morphism. 

\begin{proposition}\cite[Proposition 7.17]{donoso_maass_radic2023partial} \label{prop very rigid family}
Fix $L \geq 6$ and let $(X_{\zeta_{L}}, \cB, \nu, S)$ be the substitution subshift given by $\zeta_L \colon \cA^* \to \cA^*$,  then

\begin{equation*}
   \delta_{\nu} = \nu(aa) + \nu(bb) = \max_{k\geq 2  }  \left\{ \sum_{w \in \cC \cA^k}  \nu (w) ,\sum_{w \in \overline{\cC} \cA^k}  \nu (w)  \right\} = \frac{L-1}{L+1}
\end{equation*}

\end{proposition}

Now we can give a detailed version of \cref{main thrm} stated in the introduction. For that, as for \cref{cor one substitution}, we write $\zeta_L \colon \cA_i^* \to \cA_i^*$ even if it is originally define in the alphabet $\cA$.

\begin{theorem} \label{thrm final result}
    For $L \geq 6$, let $\boldsymbol \sigma $ be the directive sequence of glued substitutions $ \boldsymbol \sigma = ( \Gamma(\zeta_{L^{2^{i+1}}}^{(n+1)2^{d-i}} \colon i =0, \ldots,d-1))_{n \in \N}$. That is 
    \begin{equation*}
        \begin{array}{cc}
             \sigma_n(a_i) &= \kappa(\zeta_{L^{2^{i+1}}}^{(n+1)2^{d-i}}(a_i))\\
        \sigma_n(b_i) &= \kappa(\zeta_{L^{2^{i+1}}}^{(n+1)2^{d-i}}(b_i)) 
        \end{array}  \quad \text{ for } i \in \{0 , \ldots, d-1\}.
    \end{equation*}
    Then,  
    \begin{equation} \label{final eq}
        \delta_{\mu_i} = \frac{L^{2^{i+1}}-1}{L^{2^{i+1}}+1}
    \end{equation}
    and the rigidity sequence is $(h^{(n)})_{n \in \N}$.
    
\end{theorem}

\begin{remark*}
    The directive sequence $\boldsymbol \sigma$ in the statement fullfils all the hypothesis of \cref{thrm gluing technique} with $\tau_i = \zeta_{L^{2^{i+1}}}^{2^{d-i}} $ for $i = 0, \ldots, d-1$. In particular, $|\zeta_{L^{2^{i+1}}}^{2^{d-i}}(a_i)|= (L^{2^{i+1}})^{2^{d-i}} = L^{2^{i+1} \cdot 2^{d-i}} = L^{2^{d+1}}$, therefore the morphism $\sigma_n$ is indeed of constant length, for all $n \in \N$. In this case $h^{(n)} = (L^{2^{d+1}})^{n+1}\cdot(L^{2^{d+1}})^{n} \cdots L^{2^{d+1}}$. 
\end{remark*}

\begin{proof}
    By \cref{prop very rigid family} 
    \begin{equation*}
        \max_{k=2, \ldots, L^{2^{d+1}}  }  \left\{  \nu (\cC \cA_i^k) ,  \nu ( \overline{\cC}\cA_i^k)  \right\}   = \nu_i (a_ia_i) + \nu_i (b_ib_i)= \frac{L^{2^{i+1}}-1}{L^{2^{i+1}}+1} = \delta_{\nu_i}.
    \end{equation*}
    Therefore, by \cref{cor delta smaler} and \cref{thrm combination bound max}, 
    $ \displaystyle \delta_{\mu_i} = \delta_{\nu_i}$, concluding \eqref{final eq}. Since 
 \begin{equation*}
     \lim_{n \to \infty} \sum_{j=0}^{d-1} \mu_i^{(n)}(a_ja_j) + \mu_i^{(n)}(b_jb_j) = \lim_{n \to \infty} \mu_i^{(n)}(a_ia_i) + \mu_i^{(n)}(b_ib_i) = \delta_{\mu_i},
 \end{equation*}
 by \cref{theorem constant length delta mu}, the partial rigidity sequence is given by $(h^{(n)})_{n \in \N}$.
\end{proof}

\begin{remark*}
    The construction in \cref{thrm final result} relies on the fact that the family of substitutions $\zeta_L$, for $L\geq 6$, had been previously studied.  A similar result  could be achieved including $\zeta_2$, which corresponds to the Thue-Morse substitution, also studied in \cite{donoso_maass_radic2023partial}, but in that case the partial rigidity sequence for its corresponding ergodic measure would have been $(3 \cdot h^{(n)})_{n \in \N}$ instead of $(h^{(n)})_{n \in \N}$. 
    In general, studying the partial rigidity rates of more substitution subshifts should allow us to construct more examples like the one in \cref{thrm final result}. In particular, it would be interesting to construct systems with algebraically independent partial rigidity rates, but even in the uniquely ergodic case, there is no example in the literature of a system with an irrational partial rigidity rate.

    We also notice that with the gluing technique introduced in \cref{thrm gluing technique} one can only build constant-length $\cS$-adic subshift, which have non-trivial rational eigenvalues, that is $m \geq 2$ and $1\leq k <m$ such that for some non-zero function $f \in L^2(X, \mu)$, $f \circ S = e^{2\pi i k /m} f$. Thus, a refinement of \cref{main thrm} would be to construct a minimal system with distinct weak-mixing measures and distinct partial rigidity rates. To construct an explicit example following a similar approach to the one outlined in this paper, it would be necessary to use a non-constant-length directive sequence and then being forced to use the general formula for the partial rigidity rate from \cite[Theorem B]{donoso_maass_radic2023partial}. Additionally, the equation \eqref{eq formula1} no longer holds in the non-constant-length case.

\end{remark*}

\bibliographystyle{abbrv} 


\bibliography{refs} 

\end{document}